\documentclass[a4paper,11pt]{article}
\usepackage{amsfonts,amssymb,amsthm,cite,amsmath,amstext}

\usepackage{paralist}
\usepackage{color}
\usepackage[colorlinks,linkcolor=red,citecolor=blue]{hyperref}

\numberwithin{equation}{section}

\definecolor{plum}{rgb}{0.8,0.2,0.8}

\title{\textbf{Non-Degeneracy of Kobayashi Volume Measures for Singular Directed Varieties}}
\author{\textsc{Ya Deng}}
\date{}
\begin{document}
\maketitle

\theoremstyle{definition}
\newtheorem*{pf}{Proof}
\newtheorem{thm}{Theorem}[section]
\newtheorem{rem}{Remark}[section]
\newtheorem{problem}{Problem}[section]
\newtheorem{conjecture}{Conjecture}[section]
\newtheorem{lem}{Lemma}[section]
\newtheorem{cor}{Corollary}[section]
\newtheorem{dfn}{Definition}[section]
\newtheorem{proposition}{Proposition}[section]
\newtheorem{example}{Example}[section]

\def\oc{\mathcal{O}} \def\oce{\mathcal{O}_E} \def\xc{\mathcal{X}}
\def\ac{\mathcal{A}} \def\rc{\mathcal{R}} \def\mc{\mathcal{M}}
\def\wc{\mathcal{W}} \def\fc{\mathcal{F}} \def\cf{\mathcal{C_F^+}}
\def\jc{\mathcal{J}}
\def\ic{\mathcal{I}}
\def\kc{\mathcal{K}}
\def\cb{\mathbb{C}}
\def\vol{\operatorname{vol}}
\def\ord{\operatorname{ord}}
\def\Im{\operatorname{Im}}
\def\dm{\mathrm{d}}

\def\as{{a^\star}} \def\es{e^\star}

\def\tl{\widetilde} \def\tly{{Y}} \def\om{\omega}

\def\cb{\mathbb{C}} \def\nb{\mathbb{N}} \def\nbs{\mathbb{N}^\star}
\def\pb{\mathbb{P}} \def\pbe{\mathbb{P}(E)} \def\rb{\mathbb{R}}
\def\zbb{\mathbb{Z}}

\def\hb{\bold{H}} \def\fb{\bold{F}} \def\eb{\bold{E}}
\def\pbb{\bold{P}}

\def\nab{\overline{\nabla}} \def\n{|\!|} \def\spec{\textrm{Spec}\,}
\def\cinf{\mathcal{C}_\infty} \def\d{\partial}
\def\db{\overline{\partial}}
\def\hess{\sqrt{-1}\partial\!\overline{\partial}}
\def\zb{\overline{z}} \def\lra{\longrightarrow}

\begin{abstract}
In this note, we prove the \emph{generic} Kobayashi volume measure hyperbolicity of singular directed varieties $(X,V)$, as soon as the canonical sheaf $\mathcal{K}_V$ of $V$ is big in the sense of Demailly. 
\end{abstract}

\section{Introduction}

Let $(X,V)$ be a \emph{complex directed manifold}, i.e $X$ is a complex manifold equipped with a holomorphic subbundle $V\subset T_X$. Demailly's philosophy in introducing directed manifolds is that, there are certain fonctorial constructions which work better in the category of directed manifolds (ref. \cite{Dem12}), even in the ``absolute case", i.e. the case $V=T_X$. Therefore, it is usually inevitable to allow singularities of $V$, and $V$ can be seen as a coherent subsheaf of $T_X$ such that $T_X/V$ is torsion free. In this case $V$ is a subbundle of $T_X$ outside an analytic subset of codimension at least $2$, which is denoted by $\text{Sing}(V)$. The Kobayashi volume measure can also be defined for (singular) directed manifolds.
\begin{dfn}\label{Kobayashi volume}
	Let $(X,V)$ be a directed manifold with $\text{dim}(X)=n$ and $\text{rank}(V)=r$. Then the Kobayashi volume measure of $(X,V)$ is the pseudometric defined on any $\xi\in \wedge^rV_x$ for $x\notin \text{Sing}(V)$, by
	$$
\mathbf{e}^r_{X,V}(\xi):=\text{inf}\{\lambda>0;\ \exists f:\mathbb{B}_r\rightarrow X,\ f(0)=x,\ \lambda f_{*}(\tau_0)=\xi, \ f_*(T_{\mathbb{B}_r})\subset V \},
$$
where $\mathbb{B}_r$ is the unit ball in $\cb^r$ and $\tau_0=\frac{\d}{\d t_1}\wedge \cdots \frac{\d}{\d t_r}$ is the unit $r$-vector of $\cb^r$ at the origin. We say that $(X,V)$ is \emph{generic volume measure hyperbolic} if $\mathbf{e}^r_{X,V}$ is generically positive definite.
\end{dfn}

In \cite{Dem12} the author also introduced the definition of \emph{canonical sheaf} $\mathcal{K}_V$ for any singular directed variety $(X,V)$, and he showed that the ``bigness" of $\mathcal{K}_V$ will imply that, any non-constant entire curve $f:\cb \rightarrow (X,V)$ must satisfy certain global
algebraic differential equations. In this note, we will study the Kobayashi volume measure of the singular directed variety $(X,V)$, and give another intrinsic explanation for the bigness of $\mathcal{K}_V$. Our main theorem is as follows:

\begin{thm}\label{main}
	Let $(X,V)$ be a compact complex directed variety ($V$ is singular) with $\text{rank}(V)=r$ and $\text{dim}(X)=n$. If $V$ is of general type (see Definition \ref{general} below), with the base locus  $\text{Bs}(V)\subsetneq X$ (see also Definition \ref{general}), then Kobayashi volume measure of $(X,V)$ is non-degenerate outside $\text{Bs}(V)$. In particular $(X,V)$ is generic volume measure  hyperbolic.
\end{thm}

\begin{rem}
	In the absolute case, Theorem \ref{main} is proved in \cite{Gri71} and \cite{KO71}; for the smooth directed variety it is proved in \cite{Dem12}.
\end{rem}

\section{Proof of main theorem}

\begin{proof}
	Since the singular set $\text{Sing}(V)$ of $V$  is an analytic set of codimension $\geq 2$, the top exterior power $\wedge ^r V$ of $V$ is a coherent sheaf of rank 1, and is included in its bidual $\wedge ^r V^{**}$, which is an invertible sheaf (of course, a line bundle). We will give an explicit construction of the multiplicative cocycle which represent the line bundle
	$\wedge ^r V^{**}$.
	
	Since $V\subset T_X$ is a coherent sheaf, we can take an open covering $\{U_\alpha\}$ satisfying that on each $U_\alpha$ one can find sections $e^{(\alpha)}_1,\ldots,e^{(\alpha)}_{k_\alpha}\in \Gamma(U_\alpha,T_X|_{U_\alpha})$ which generate the coherent sheaf $V$ on $U_\alpha$. Thus the sections  $e^{(\alpha)}_{i_1}\wedge \cdots \wedge e^{(\alpha)}_{i_r}\in \Gamma(U_\alpha,\wedge^rT_X|_{U_\alpha})$ with $(i_1,\ldots,i_r)$ varying among all $r$-tuples of $(1,\ldots,k_\alpha)$ generate the coherent sheaf $\wedge ^r V|_{U_\alpha}$, which is a subsheaf of $\wedge^rT_X|_{U_\alpha}$.  Let $v^{(\alpha)}_I:=e^{(\alpha)}_{i_1}\wedge \cdots \wedge e^{(\alpha)}_{i_r}$, then by $\text{Cod}(\text{Sing}(V))\geq2$ we know that the common zero of all $v^{(\alpha)}_I$ is contained in $\text{Sing}(V)$, and thus all $v^{(\alpha)}_I$ are proportional outside $\text{Sing}(V)$. Therefore there exists a unique $v_\alpha\in\Gamma(U_\alpha,\wedge^rT_X|_{U_\alpha})$,  and holomorphic functions $\{\lambda_I\}$ which do not have common factors,  such that  $v^{(\alpha)}_I=\lambda_I v_\alpha$ for all $I$. By this construction we can show that on $U_\alpha\cap U_\beta$, $v_\alpha$ and $v_\beta$
	coincide up to multiplication by a nowhere vanishing holomorphic function, i.e.
	$$
	v_\alpha=g_{\alpha\beta}v_\beta
	$$
	on $U_\alpha\cap U_\beta\neq \emptyset$, where $g_{\alpha\beta}\in \oc^*_X(U_\alpha\cap U_\beta)$. This multiplicative cocycle $\{g_{\alpha\beta}\}$ corresponds to the line bundle $\wedge^r V^{**}$. Then fix a K\"ahler metric $\omega$ on $X$, it will induce a metric $H_r$ on $\wedge^rT_X$ and thus also induce a singular hermitian metric $h_s$ of $\wedge^r V^{***}$ whose local weight $\varphi_\alpha$ is equal to $\log |v_\alpha|^2_{H_r}$. It is easy to show that $h_s$ has \emph{analytic singularities}, and the set of singularities $\text{Sing}(h_s)= \text{Sing}(V)$. Indeed, we have $\text{Sing}(h_s)=\cup_{\alpha}\{p\in U_\alpha| v_\alpha(p)=0 \}$. Now we make the following definition.
	
	\begin{dfn}\label{general}
		With the notions above, $V$ is called to be \emph{of general type} if there exists a singular metric $h$ on $\wedge ^r V^{***}$ with analytic singularities satisfying the following two conditions:
		\begin{enumerate}[(1)]\label{property}
			\item The curvature current $\Theta_h\geq \epsilon \omega$, i.e., it is a K\"ahler current.
			\item $h$ is more singular than $h_s$, i.e., there exists a globally defined quasi-psh function $\chi$ which is bounded from above such that 
			$$
			e^{\chi} \cdot h=h_s.
			$$
		\end{enumerate}
	Since $h$ and $h_s$ has both analytic singularities, $\chi$ also has analytic singularities, and thus $e^{\chi}$ is a continuous function. Moreover, $e^{\chi(p)}> 0$ if $p\notin \text{Sing}(h)$. We define \emph{the base locus of $V$} to be
	$$
	\text{Bs}(V):= \cap_h\text{Sing}(h),
	$$
	where $h$ varies among all singular metric on $\wedge^r V^{***}$
 satisfying the Properties (1) and (2) above.		
	\end{dfn}
	Now fix a point $p\notin \text{Bs}(V)$,  then by Definition \ref{general} we can find a singular metric $h$ with analytic singularities satisfying the Property (1) and (2) above, and $p\notin \text{Sing}(h)$.

	Let $f$ be any holomorphic map from the unit ball $\mathbb{B}_r\subset \cb^r$ to $(X,V)$ such that $f(0)=p$, then locally we have
	$$
	f_*(\frac{\d}{\d t_1}\wedge\cdots\wedge \frac{\d}{\d t_r})=a^{(\alpha)}(t)\cdot v_\alpha|_{f},
	$$
	where $a^{(\alpha)}(t)$ is meromorphic functions, with poles contained in $f^{-1}(\text{Sing}(V))$, and satisfies
	$$
	|f_*(\frac{\d}{\d t_1}\wedge\cdots\wedge \frac{\d}{\d t_r})|^2_{\omega}=|a^{(\alpha)}(t)|^2\cdot |v_\alpha|^2_{H_r}\leq C.
	$$
	Since $f_*(\frac{\d}{\d t_1}\wedge\cdots\wedge \frac{\d}{\d t_r})$ can be seen as a (meromorphic!) section of $f^*\wedge^r V^{**}$, then we define
\begin{equation}\nonumber
\delta(t):=|f_*(\frac{\d}{\d t_1}\wedge\cdots\wedge \frac{\d}{\d t_r})|^2_{h^{-1}}=|a^{(\alpha)}(t)|^2\cdot e^{\phi_\alpha},
\end{equation}
	where $\phi_\alpha$ is the local weight of $h$. By Property (2) above, we have a globally defined quasi-psh function $\chi$ on $X$ which is bounded from above such that 
\begin{equation}\label{1}
	\delta(t)=e^{\chi}\cdot |f_*(\frac{\d}{\d t_1}\wedge\cdots\wedge \frac{\d}{\d t_r})|^2_{\omega}\leq C_1.
\end{equation}
	
	Now we define a semi-positive metric $\widetilde{\gamma}$ on $\mathbb{B}_r$ by putting $\widetilde{\gamma}:=f^{*}\omega$, then we have
	\begin{equation}\label{2}
	\frac{|f_*(\frac{\d}{\d t_1}\wedge\cdots\wedge \frac{\d}{\d t_r})|_{\omega}}{\det \widetilde{\gamma}}\leq C_0(f(t)),
	\end{equation}
	where $C_0(z)$ is a bounded function on $X$ which does not depend on $f$, and we take $C_2$ to be its upper bound. One can find a conformal $\lambda(t)$ to define $\gamma=\lambda \widetilde{\gamma}$ satisfying
	$$
	\det \gamma=\delta(t)^{\frac{1}{2}}.
	$$
Combined (\ref{1}) and (\ref{2}) togother we obtain
$$
\lambda\leq C_2^{\frac{1}{r}}\cdot e^{\frac{\chi}{2r}}.
$$
Since $\Theta_h\geq \epsilon \omega$, by (\ref{1}) we have
	$$
	\dm \dm^c\log \det\gamma \geq \frac{\epsilon}{2}f^{*}\omega=\frac{\epsilon}{2\lambda}\gamma\geq \frac{\epsilon}{2C_2^{\frac{1}{r}}} e^{-\frac{\chi\circ f}{2r}}\gamma.
	$$
By Property (2) in Definition \ref{general} of $h$, there exists a constant $C_3>0$ such that
$$
e^{-\frac{\chi}{2r}}\geq C_3.
$$
Let $A:=\frac{\epsilon C_3}{2C_2^{\frac{1}{r}}}$, and we know that it is a constant which does not depend on $f$. Then by Ahlfors-Schwarz Lemma (see Lemma \ref{Ahlfors} below) we have
	$$
	\delta(0)\leq(\frac{r+1}{A})^{2r}.
	$$
	Since $p\notin \text{Sing}(h)$, then we have $e^{\chi(p)}>0$,  and thus
	$$
	|f_*(\frac{\d}{\d t_1}\wedge\cdots\wedge \frac{\d}{\d t_r})|^2_{\omega}(0)\leq e^{-\chi(p)}\delta(0)=e^{-\chi(p)}\cdot(\frac{r+1}{A})^{2r}.
	$$
	By the arbitrariness of $f$, and Definition \ref{Kobayashi volume}, we conclude that $(X,V)$ is generic volume measure hyperbolic and $\mathbf{e}^r_{X,V}$ is non-degenerate outside $\text{Bs}(V)$.
	\end{proof}
	
	\begin{lem}[Ahlfors-Schwarz]\label{Ahlfors}
	Let $\gamma=\sqrt{-1}\sum \gamma_{jk}(t)\dm t_j\wedge \dm t_k$ be an almost everywhere positive hermitian form on the ball  $B(0,R)\subset \cb^r$ of radius $R$, such that
	$$
	-\text{Ricci}(\gamma):=\sqrt{-1}\d \db \log \det \gamma \geq A \gamma 
	$$
	in the sense of currents for some constant $A>0$. Then
	$$
	\det(\gamma)(t)\leq (\frac{r+1}{AR^2})^r \frac{1}{(1-\frac{|t|^2}{R^2})^{r+1}}.
	$$
	\end{lem}
	
	\begin{rem}
		If $V$ is regular, then $V$ is of general type if and only if $\wedge^r V$ is a big line bundle. In this situation, the base locus $\text{Bs}(V)=\mathbf{B}_+(\wedge^r V)$, where $\mathbf{B}_+(\wedge^r V)$ is the \emph{augmented base locus} for the big line bundle $\wedge^r V$ (ref. \cite{Laz04}).
	\end{rem}

	With the notions above, we define the coherent ideal sheaf $\ic(V)$ to be germ of holomorphic functions which is locally bounded with respect to $h_s$, i.e., $\ic(V)$ is the integral closure of the ideal generated by the coefficients of $v_\alpha$ in some local trivialization in $\wedge^r T_X$. Let $\wedge^r V^{***}$ be denoted by $K_V$, and $\mathcal{K}_V:=K_V\otimes \ic(V)$, then $\mathcal{K}_V$ is the \emph{canonical sheaf of $(X,V)$} defined in \cite{Dem12}. It is easy to show that the zero scheme of $\ic(V)$ is equal to $\text{Sing}(h_s)=\text{Sing}(V)$.  $\mathcal{K}_V$ is called to be a \emph{big sheaf} iff for some log resolution $\mu:\widetilde{X}\rightarrow X$ of $\ic(V)$ with $\mu^*\ic(V)=\oc_{\widetilde{X}}(-D)$, $\mu^*K_V-D$ is big in the usual sense.  Now we have the following statement:
	
	\begin{proposition}\label{equiv}
	$V$ is of general type if and only if $\mathcal{K}_V$ is big.	Moreover, we have
	$$
	\text{Bs}(V)\subset \mu(\mathbf{B}_+(\mu^*K_V-D))\cup \text{Sing}(h_s)= \mu(\mathbf{B}_+(\mu^*K_V-D))\cup \text{Sing}(V).
	$$
	\end{proposition}  
	
	\begin{proof}
		From Definition \ref{general}, the condition that $\mathcal{K}_V$ is a big sheaf implies that $K_V$ and $\mu^*K_V-D$ are both big line bundles. For $m\gg0$, we have the following isomorphism
\begin{equation}\label{iso}
	\mu^*:H^0(X,(mK_V-A)\otimes \ic(V)^m)\xrightarrow{\approx} H^0(\widetilde{X},m\mu^*{K_V}-\mu^*A-mD).
\end{equation}
		Fix an very ample divisor $A$, then for $m\gg0$, the base locus (in the usual sense) $\mathbf{B}(m\mu^*{K_V}-mD-\mu^*A)$ is stably contained in $\mathbf{B}_+(\mu^*K_V-D)$ (ref. \cite{Laz04}). Thus we can take a $m\gg0$ to choose a basis $s_1,\ldots,s_k\in H^0 (\widetilde{X},m\mu^*{K_V}-mD-\mu^*A)$, whose common zero is contained in $\mathbf{B}_+(\mu^*K_V-D)$. Then by (\ref{iso}) there exists $\{e_i\}_{1\leq i\leq k}\subset H^0(X,(mK_V-A)\otimes \ic(V)^m)$ such that 
		$$
		\mu^*(e_i)=s_i.
		$$
		Then we can define a singular metric $h_m$ on $mK_V-A$ by putting $|\xi|^2_{h_m}:=\frac{|\xi|^2}{\sum_{i=1}^{k}|e_i|^2}$ for $\xi\in (mK_V-A)_x$. We take a smooth metric $h_A$ on $A$ such that the curvature $\Theta_A\geq \epsilon\omega$ is a smooth K\"ahler form. Then $h:=(h_mh_A)^{\frac{1}{m}}$ defines a singular metric on $K_V$ with analytic singularities, such that its curvature current $\Theta_h\geq \frac{1}{m}\Theta_A$. From the construction we know that $h$ is more singular than $h_s$, and $\text{Sing}(h)\subset \mu(\mathbf{B}_+(\mu^*K_V-D))\cup \text{Sing}(h_s)$. 
	\end{proof}
	
	\begin{rem}
		From Proposition \ref{equiv} we can take Definition \ref{general} as another equivalent definition of the bigness of $\mathcal{K}_V$, and it is more analytic. From Theorem \ref{main} we can replace the condition that $V$ is of general type by the bigness of $\mathcal{K}_V$, and it means that the definition of canonical sheaf of singular directed varieties is very natural.
	\end{rem}
	
A direct consequence of Theorem \ref{main} is the following corollary, which was suggested in \cite{GPR13}: 
	\begin{cor}
		Let $(X,V)$ be directed varieties with $\text{rank}(V)=r$, and $f$ be a holomorphic map from $\cb^r$ to $(X,V)$ with generic maximal rank. Then if $\mathcal{K}_V$ is big, the image of $f$ is contained in $\text{Bs}(V)\subsetneq X$.
	\end{cor}
	
The famous conjecture by Green-Griffiths stated that in the absolute case the converse of Theorem \ref{main} should be true. It is natural to ask whether we have similar results for any directed varieties.  A result by Marco Brunella (ref. \cite{Bru10}) gives a weak converse of Theorem \ref{main} for 1-directed variety:

\begin{thm}
Let $X$ be a compact K\"ahler manifold equipped with a singular holomorphic foliation $\fc$ by curves. Suppose that $\fc$ contains at least one leaf which is hyperbolic, then the canonical bundle $K_\fc$ is pseudoeffective.
\end{thm}

Indeed, Brunella proved more than the results stated in the theorem above. By putting on $K_\fc$ precisely the Poincar\'e metric on the hyperbolic leaves, he construct a singular hermitian metric $h$ (maybe not with analytic singularities) on $K_\fc$, such that the set of points where $h$ is unbounded locally are polar set $\text{Sing}(\fc)\cup \text{Parab}(\fc)$, where $\text{Parab}(\fc)$ are the union of parabolic leaves, and the curvature $\Theta_h$ of the metric $h$ is a positive current. Therefore, it seems that Brunella's theorem can be strengthened, i.e., not only $K_\fc$ is pseudo-effective, but also the canonical sheaf $\mathcal{K}_\fc$ is pseudoeffective. However, as is shown in the following example, even if all the leaves of $\fc$ are hyperbolic, the canonical sheaf can not be pseudoeffective.

\begin{example}\label{exam}
A foliation by curves of degree $d$ on the complex projective space $\cb P^n$ is generated by a global section
$$s\in H^0(\cb P^n, T_{\cb P^n} \otimes \oc(d-1)).$$
From the results by Lins Neto
and Soares \cite{LS96} and Brunella \cite{Bru06}, we know that a \emph{generic} one-dimensional foliation $\fc$ of degree $d$ satisfies the following property:
\begin{enumerate}[(a)]
\item the set of the singularities $\text{Sing}(\fc)$ of $\fc$ is discrete;
\item each singularity $p\in \text{Sing}(\fc)$ is non-degenerate, i.e. the Milnor number of $\fc$ at $p$ is 1;
\item no $d+1$ points in $\text{Sing}(\fc)$ lie on a projective line;\label{non linear}
\item all the leaves of $\fc$ are hyperbolic.
\end{enumerate}  Hence the Baum-Bott formula states that
\begin{eqnarray}\nonumber
\#\text{Sing}(\fc)&=&c_n(T_{\cb P^n}+\oc(d-1))\\\nonumber
&=&\sum_{i=0}^{n}c_{n-i}(\cb P^n)c_1(\oc(d-1))^i\\\nonumber
&=&\sum_{i=0}^{n}\dbinom{n+1}{i+1}(d-1)^i\\\nonumber
&=&\frac{d^{n+1}-1}{d-1},
\end{eqnarray}
and thus the canonical sheaf $\mathcal{K}_\fc=\oc(d-1)\otimes \mathcal{I}_{\text{Sing}(\fc)}$, where $\ic_{\text{Sing}(\fc)}$ is the maximal ideal of $\text{Sing}(\fc)$. By property (\ref{non linear}) above it is easy to prove that for $d\gg 0$ $\mathcal{K}_\fc$ is not pseudo-effective.
\end{example}

\begin{rem}
In \cite{McQ08} the author introduces the definition \emph{canonical singularities} for foliations, in dimension 2 this definition is equivalent to \emph{reduced singularities} in the sense of Seidenberg. The generic foliation by curves of degree $d$ in $\cb P^n$ is another example of canonical singularities. In this situation, one can not expect to improve the ``bigness" of canonical sheaf $\mathcal{K}_\fc$ by blowing-up. Indeed, this birational model is ``stable" in the sense that, $\pi_*\mathcal{K}_{\widetilde{\fc}}=\mathcal{K}_\fc$ for any birational model $\pi:(\widetilde{X},\widetilde{\fc})\rightarrow(X,\fc)$. However, on the complex surface endowed with a foliation $\fc$ with reduced singularities, if $f$ is an entire curve tangent to the foliation, and $T[f]$ is the Ahlfors current associated with $f$, then in \cite{McQ98} it is shown that the positivity of $T[f]\cdot T_\fc$ can be improved by an infinite sequence of blowing-ups, due to the fact that certain singularities of $\fc$ appearing in the blowing-ups are ``\emph{small}", i.e. the lifted entire curve will not pass to these singularities. Since $T[f]\cdot T_\fc$ is related to value distribution, and thus these small singularities are negligible. In \cite{Den16} this ``Diophantine approximation" idea has been generalized to higher dimensions. 
\end{rem}

\noindent

\textbf{Acknowledgements}
I would like to thank my supervisor Professor Jean-Pierre Demailly for stimulating exchanges and comments in this work. The author is supported by the China Scholarship Council.

\textsc{Institut Fourier \& University of Science and Technology of China} \\
\textsc{Ya Deng} \\
\verb"Email: Ya.Deng@fourier-grenoble.fr"

\end{document}